\documentclass[12pt,leqno]{article}
\usepackage[shortlabels]{enumitem}
\usepackage{preamble}

\title{Positively curved Finsler metrics on vector bundles III \thanks{Mathematics Subject Classification 32J25, 32F17. \newline Keywords: Kobayashi positivity, Griffiths positivity, ampleness.}}
\author{Kuang-Ru Wu}
\newcommand{\RN}[1]{%
  \textup{\uppercase\expandafter{\romannumeral#1}}%
}

\usepackage{fullpage}
\usepackage{tikz-cd}

\usepackage{dutchcal}

\usepackage{fancyhdr}

\theoremstyle{plain}

\numberwithin{equation}{section}

\begin{document}

\date{}

\parskip=6pt

\maketitle
\begin{abstract}
The goal of the paper is to extend results about ample or Griffiths positive vector bundles to Kobayashi positive vector bundles. In particular, we show that the quotient bundle of a Kobayashi positive vector bundle is Kobayashi positive, and the tensor product of two Kobayashi positive vector bundles is Kobayashi positive. 

These results strengthen the conjectural equivalences between ampleness, Griffiths positivity, and Kobayashi positivity. The proofs rely on the convexity of Kobayashi positive Finsler metrics and the duality for convex Finsler metrics. 
\end{abstract}

\section{Introduction}
In \cite{Hart66}, Hartshorne proved several properties regarding ampleness of vector bundles; for example, taking quotient and taking tensor product preserve ampleness. Most of these properties are also valid for Griffiths positive Hermitian metrics (\cite{Griff69,Umemura,demailly1997complex}). In light of the conjectures of Griffiths and Kobayashi which predict that ampleness, Griffiths positivity, and Kobayashi positivity should all be equivalent, one is led to the question whether these properties are valid for Kobayashi positive Finsler metrics. The main purpose of the paper is to address this question. (For the progress on the conjectures of Griffiths and Kobayashi, see \cite{Umemura,CampanaFlenner,Berndtsson09,MourouganeTaka,toric,positivityandvanishingthmliu,naumann2017approach,FengLiuWan,demailly2020hermitianyangmills,pingali2021note,wu_2022,wupositivelyII,Mazhang}).

We use $E, E_1$, and $E_2$ for holomorphic vector bundles and $X, Y$ for the base compact complex manifolds. A vector bundle is called Kobayashi positive if it carries a strongly pseudoconvex Finsler metric with positive Kobayashi curvature (we will review Finsler metrics in Subsection \ref{subsec finsler}). Our main results are the following.
\begin{theorem}\label{thm 1}
\
\begin{enumerate}
    \item\label{1} The quotient bundle of a Kobayashi positive vector bundle is Kobayashi positive.    
    \item If $E_1$ and $E_2$ are Kobayashi positive, then $E_1\otimes E_2$ is Kobayashi positive.
    \item If $E$ is Kobayashi positive, then the tensor power $E^{\otimes k}$ and the symmetric power $S^kE$ are Kobayashi positive for $k\geq 1$, and the exterior power $\bigwedge^k E$ is Kobayashi positive for $1\leq k \leq r$ where $ r=\text{rank } E  $.
\end{enumerate}

\end{theorem}

The second statement in Theorem \ref{thm 1} is perhaps most interesting. It is based on a general theorem in \cite[Theorem 2.4]{LLmax} (similar results in various settings can be found in \cite[Theorem 4.2]{rochberg1984}, \cite[Corollary 4.5]{SlodI}, and \cite[Proposition 15.5]{CofimanSemmes} where the focus is on the interpolation problem). This theorem, roughly speaking, says that if a holomorphic homomorphism between Banach bundles decreases curvature, then the logarithm of the operator norm of the homomorphism is plurisubharmonic. 

Since every homomorphism from $E_1$ to $E_2^*$ decreases curvature provided that $E_1$ is positively curved and $E_2^*$ is negatively curved, the bundle $\Hom(E_1,E_2^*)$ with the operator norm is negatively curved according to the Lempert's theorem above, so $E_1^*\otimes E_2^*$ is negatively curved. This is basically the idea behind the proof, but we cannot simply apply the results in \cite{rochberg1984,SlodI,CofimanSemmes,LLmax} to our case. One main catch is that our predecessors all assume convexity on the metrics, whereas the Finsler metrics we consider are not necessarily convex.
Another obstacle concerns dual Finsler metrics. In general, taking duality of Finsler metrics does not flip the sign of curvature (\cite[Remark 2.7]{DemaillyMSRI}); however, for convex Finsler metrics, we do have the duality (\cite{Sommese,DemaillyMSRI} or see Lemma \ref{sommese}. In \cite[Lemma 2.1]{LLextrapolation}, Lempert proved duality for Banach bundles). All these issues can be resolved by our result in \cite[Corollary 2]{wu_2022}: when it comes to Kobayashi positive Finsler metrics, one has convexity for free. The use of our convexity result from \cite[Corollary 2]{wu_2022} is the crux of the paper.

Another feature which does not seem to appear in the literature before is regularization of continuous Finsler metrics. We prove two regularization lemmas in Subsection \ref{sec regularize} based on \cite{GreeneWu2,GreeneWu}.   

Besides our application here, the theorem in \cite{LLmax} is also used in \cite{LLextrapolation} to prove an Ohsawa--Takegoshi type extension theorem \cite{BoLem}, in  \cite{LLnoncommutative,wuDirichlet} to solve a Dirichlet problem for flat metrics on Hilbert bundles, and in \cite{albesiano2023deformation} to prove a Skoda type division theorem.

In \cite{wupositivelyII}, as a motivation, we mentioned that if $E$ is ample (Griffiths positive) then $E^*\otimes \det E$ is ample (Griffiths positive), but we did not know whether the statement holds for Kobayashi positivity. We can answer this question here by using the statement 3 in Theorem \ref{thm 1} together with the isomorphism $\bigwedge^{r-1}E\simeq 
E^*\otimes \det E $; therefore, if $E$ is Kobayashi positive then so is $E^*\otimes \det E$. A related question is that if $E$ is ample, can we prove Kobayashi positivity of $E^*\otimes \det E$? Since $E\otimes \det E^* \simeq \Hom (E^*, \det E^*)$ and $\det E^*$ is theoretically more negative than $E^*$, it is tempting to use Lempert's theorem on curvature decreasing, but we have not been able to compare the curvature of $E^*$ and $\det E^*$.

Let us discuss a naive approach to Theorem \ref{thm 1}. It is known that ampleness of $E$ is equivalent to Kobayashi negativity of $E^*$ (\cite[Theorem 5.1]{Negfinsler}), and that Theorem \ref{thm 1} is true if Kobayashi positivity is replaced with ampleness (\cite{Hart66}), so Theorem \ref{thm 1} is true for Kobayashi negativity (quotient bundle replaced with subbundle in the statement 1). However, one cannot simply take the dual Finsler metrics to conclude Theorem \ref{thm 1} because, again, taking duality does not in general flip the sign of the curvature for Finsler metrics. 

There are some results about tensor products and Kobayashi negativity. In \cite{Negfinsler}, Kobayashi asked whether one can prove the fact that Kobayashi negativity is preserved under tensor product by using a differential geometric method. The statements in \cite[Theorem 5.6 and Corollary 5.7]{DemaillyMSRI} can be viewed as an answer to Kobayashi's question. In \cite{BenAbdesselem}, Ben Abdesselem showed that if $E_1$ is Kobayashi negative  and $E_2^*$ is globally generated by holomorphic sections, then $E_1\otimes E_2$ is Kobayashi negative (see also \cite{BenAbdesselem2}). However, their results do not imply ours because one cannot take duality due to the lack of convexity. 

In addition to Theorem \ref{thm 1}, we collect a few more properties of Kobayashi positivity in Theorem \ref{thm 2} below. We use the strategy developed in the proof of Theorem \ref{thm 1} to prove them; pass to the dual bundles, manipulate plurisubharmonicity, and then take dual again using convexity. It seems likely the results in Theorem \ref{thm 2} can also be proved using curvature formulas from \cite{Negfinsler, ComplexFinsler} without passing to the dual bundles.

\begin{theorem}\label{thm 2}
\
\begin{enumerate}
    \item The direct sum of vector bundles $E_1\oplus E_2$ is Kobayashi positive if and only if $E_1$ and $E_2$ are both Kobayashi positive.
    \item In the short exact sequence of vector bundles $0\to E_1\to E \to E_2\to 0$, if $E_1$ and $E_2$ are Kobayashi positive, then $E$ is Kobayashi positive.
    \item\label{3} If $f:Y\to X$ is an immersion, and $E$ is a Kobayashi positive vector bundle over $X$, then the pull-back $f^*E$ is Kobayashi positive.
\end{enumerate}
\end{theorem}

Although this paper is the third in our sequence \cite{wu_2022,wupositivelyII}, we still include a section to review Finsler metrics, Kobayashi curvature, etc., to make this paper self-contained.

I would like to thank L\'aszl\'o Lempert for many fruitful conversations. Part of the paper was done in Taipei, and I thank Academia Sinica and National Center for Theoretical Sciences for providing a stimulating environment.

\section{Preliminaries}
\subsection{Finsler metrics}\label{subsec finsler}
We give a quick review of Finsler metrics here. For more details, we recommend \cite{Negfinsler, ComplexFinsler}. See also \cite{Aikoupartial,AikouMSRI,caowong,liuchern,liudonaldson, wu_2022}.

Let $E$ be a holomorphic vector bundle of rank $r$ over a compact complex manifold $X$ of dimension $n$. For a vector $\zeta\in E_z$, we will symbolically write $(z,\zeta)\in E$. A Finsler metric $F$ on the vector bundle $E\to X$ is a real-valued function on $E$ such that 
\begin{align*}
    &(1) \text{ $F$ is continuous on the total space $E$, and $F$ is smooth away from the zero section of $E$}.\\
    &(2) \text{ For $(z,\zeta)\in E$}, F(z,\zeta)\geq 0, \text{ and equality holds if and only if $\zeta=0$}.\\
    &(3) \text{ $F(z,\lambda\zeta)=|\lambda|F(z,\zeta)$,  for $\lambda\in \mathbb{C}$}.
\end{align*}
We denote $F^2$ by $G$ throughout the paper (we will sometimes use $G$ to denote a Finsler metric). Regarding (1), if $F$ is only continuous on $E$ without being smooth on $E\setminus \{0\}$, then we call such an $F$ a continuous Finsler metric (in this paper, Finsler metrics are always smooth, unless we write \textit{continuous} Finsler metrics). 

Denote by $P(E)$ the projectivized bundle of $E$, and by $O_{P(E)}(-1)$ the tautological line bundle over $P(E)$. Let $p$ be the projection from $P(E)$ to $X$, and denote the pull-back bundle $p^*E$ by $\tilde{E}$. In summary,
\[
 \begin{tikzcd}
  O_{P(E)}(-1) \subset  \tilde{E}   \arrow{d}\arrow{r}{\tilde{p}} & E \arrow{d}\\
   P(E) \arrow{r}{p} & X.
  \end{tikzcd}
  \]
For a vector $\zeta\in E_z$, we denote by $[\zeta]$ the equivalence class of $\zeta$ in $P(E_z)$, and we will write $(z,[\zeta])\in P(E)$. The pull-back bundle $\tilde{E}=\coprod_{(z,[\zeta])\in P(E)}(z,[\zeta])\times E_z$, and we will denote an element in $\tilde{E}$ by $(z,[\zeta],Z)$ with $Z\in E_z$. There is a one-to-one correspondence between Finsler metrics on $E$ and Hermitian metrics on $O_{P(E)}(-1)$, furnished by $\tilde{p}$.
If $F$ is a Finsler metric on $E$ with $G=F^2$ and we denote the corresponding Hermitian metric on $O_{P(E)}(-1)$ by $h_G$, then the correspondence is 
\begin{equation}\label{corres}
    h_G(z,[\zeta],\zeta)=G(z,\zeta).
\end{equation}

When we need to do local computations, we use $(z_1,...,z_n)$ for local coordinates on $X$, and $(\zeta_1,...,\zeta_r)$ for local fiber coordinates on $E$ with respect to a holomorphic frame $\{e_1,...,e_r\}$. So $(z_1,...,z_n,\zeta_1,...,\zeta_r)$ is a coordinate system on $E$. Let $F$ be a Finsler metric on $E$ with $G=F^2$. We write 
\begin{align*}
G_i=\partial G/\partial \zeta_i\,,\text{      }\text{  }G_{\bar{j}}=\partial G/\partial \bar{\zeta}_j\,,\text{      }\text{  } G_{i\Bar{j}}=\partial^2 G/\partial \zeta_i \partial\bar{\zeta}_{j}\,, \\    G_{i\alpha}=\partial G_i/\partial z_\alpha\,,\text{      }\text{  } G_{i\bar{j}\bar{\beta}}=\partial G_{i\bar{j}}/\partial \bar{z}_\beta\,,\text{ etc.,}
\end{align*}  
with Latin letters $i,j$ for the fiber direction and Greek letters $\alpha,\beta$ for the base direction. From \cite[(3.3) and (3.10)]{ComplexFinsler}, we have 
\begin{equation}\label{summmmm}
 \sum_j G_{i\Bar{j}}(z,\zeta)\Bar{\zeta}_j=G_i(z,\zeta) \text{, and } \sum_{i,j} G_{i\Bar{j}}(z,\zeta)\zeta_i\Bar{\zeta}_j=G(z,\zeta). 
\end{equation}


A Finsler metric $F$ is said to be
\begin{align*}
(1) &\text{ strongly pseudoconvex if the matrix $(F_{i\bar{j}})$ is positive definite on $E\setminus \{\text{0}\}$}.\\
(2) &\text{ convex if $F$ restricted to each fiber $E_z$ is convex},\\ &\text{ namely the fiberwise real Hessian of $F$ is positive semidefinite on $E\setminus \{\text{0}\}$}.\\
(3) &\text{ strongly convex if the fiberwise real Hessian of $G$ is positive definite on $E\setminus \{\text{0}\}$}.
\end{align*}
Note that in (1), the requirement $(F_{i\bar{j}})>0$ is equivalent to $(G_{i\bar{j}})>0$. Indeed, since $G=F^2$, we have $G_{i\bar{j}}=2F_iF_{\bar{j}}+2FF_{i\bar{j}}$ and so $(F_{i\bar{j}})>0$ implies $(G_{i\bar{j}})>0$. On the other hand, we have $$F_{i\bar{j}}=\frac{1}{4}G^{-\frac{3}{2}}(2G_{i\bar{j}}G-G_iG_{\bar{j}}    )=\frac{1}{4}G^{-\frac{3}{2}}(G_{i\bar{j}}G+G_{i\bar{j}}G-G_iG_{\bar{j}}    )$$
and if $(G_{i\bar{j}})>0$, we have $(\sum_{i,j}G_{i\Bar{j}}v_i \Bar{v}_j) G-(\sum_iG_i v_i )(\sum_j G_{\Bar{j}}\Bar{v}_j)\geq 0$ for any $v\in \mathbb{C}^r$ by (\ref{summmmm}) and the Cauchy--Schwarz inequality, so $(F_{i\Bar{j}})>0$. 

If $F_1$ and $F_2$ are two convex Finsler metrics on $E$, then the Finsler metric $(F_1^2+F_2^2)^{1/2}$ is still convex by a direct computation and the Cauchy--Schwarz inequality.

The next is a lemma characterizing strong pseudoconvexity (\cite[Theorem (1)]{ComplexFinsler}).
\begin{lemma}\label{fiber neg}
 Let $F$ be a Finsler metric on $E$ and $h_G$ be its corresponding Hermitian metric on $O_{P(E)}(-1)$. Then $F$ is strongly pseudoconvex if and only if the curvature $\Theta(h_G)$ of $h_G$ restricted to each fiber of $P(E)\to X$ is negative definite.
\end{lemma}
For a strongly pseudoconvex Finsler metric $F$, there is a natural Hermitian metric $\tilde{G}$ on the pull-back bundle $\tilde{E}$ (see \cite{Negfinsler}). In terms of local coordinates, the Hermitian metric $\tilde{G}$ is given by  
$$\tilde{G}_{(z,[\zeta])}(Z,Z)=\sum_{i,j}G_{i\bar{j}}(z,\zeta)Z_i\bar{Z}_j, \text{ for $Z=\sum^r_{i=1}Z_i s_i(z)\in E_z$}.$$
Now since $(\tilde{E},\tilde{G})$ is a Hermitian holomorphic vector bundle, we can talk about its Chern curvature $\Theta$, an $\End \tilde{E}$-valued $(1,1)$-form on $P(E)$. With respect to the metric $\tilde{G}$, the bundle $\tilde{E}$ has a fiberwise orthogonal decomposition $O_{P(E)}(-1)\oplus O_{P(E)}(-1)^\perp$, and so $\Theta$ can be written as a block matrix. Let $\Theta|_{O_{P(E)}(-1)}$ denote the block in the matrix $\Theta$ corresponding to $\End(O_{P(E)}(-1))$. Since $O_{P(E)}(-1)$ is a line bundle, $\Theta|_{O_{P(E)}(-1)}$ is a $(1,1)$-form on $P(E)$.
\begin{definition}
The $(1,1)$-form $\Theta|_{O_{P(E)}(-1)}$ is called the Kobayashi curvature of the strongly pseudoconvex Finsler metric $F$. Kobayashi positivity (or negativity) means the positivity (or negativity) of $\Theta|_{O_{P(E)}(-1)}$.
\end{definition}

We recall a local expression of the Kobayashi curvature (\cite[Formula (2.4)]{wu_2022}). We focus on a local chart $\{(z,[\zeta]):\zeta_r\neq 0\}$ of $P(E)$, and let $\zeta_i/\zeta_r=w_i$ for $1\leq i\leq  r-1$. The formula is 
\begin{equation}\label{local for koba}
  \Theta|_{O_{P(E)}(-1)}=\sum_{\alpha,\beta} \frac{\tilde{G}(R_{\alpha\bar{\beta}}\zeta,\zeta)}{\tilde{G}(\zeta,\zeta)}\,dz_\alpha \wedge d\bar{z}_\beta,
\end{equation}
where $R_{\alpha\bar{\beta}}$ is an endomorphism of $\tilde{E}$. 

We also have the following characterization (\cite[Section 5]{ComplexFinsler}).
\begin{lemma}\label{K positive}
Let $F$ be a strongly pseudoconvex Finsler metric on $E$ and $h_G$ be the corresponding Hermitian metric on $O_{P(E)}(-1)$. Then $F$ is Kobayashi positive (or negative) if and only if $\Theta(h_G)$ has signature $(n,r-1)$ (or $(0,n+r-1)$).
\end{lemma}

For a strongly pseudoconvex Finsler metric $F$, one can talk about coordinates normal at a point. Given a point $(z_0,[\zeta_0])\in P(E)$, there exists a holomorphic frame $\{e_i\}$ for $E$ around $z_0\in X$ such that 
\begin{equation}\label{normal}
 G_{\alpha}(z_0,\zeta_0)=0,\,\, G_{\alpha \bar{j} }(z_0,\zeta_0)=0
\end{equation}
(such a frame can be obtained by (5.11) in \cite{ComplexFinsler}).

Finally, we recall a formula from \cite[Section 5]{wu_2022} (see also \cite{rochberg1984,LLmax,LLextrapolation}). Let $F$ be a strongly pseudoconvex Finsler metric on $E$ with $F^2=G$. For $z_0\in X$, $v\in T^{1,0}_{z_0} X$ and $0\neq\zeta\in E_{z_0}$, we define
\begin{equation}\label{inf}
K_v(\zeta)=-\inf \partial\bar{\partial}\log G(\phi)\bigr|_{z_0}(v,\bar{v}),
\end{equation}
the inf taken over local holomorphic sections $\phi$ of $E$ such that $\phi(z_0)=\zeta$. The inf is actually minimum (see \cite[Proof of Lemma 9]{wu_2022}). Consider a tangent vector $\tilde{v}$  to $P(E)$ at $(z_0,[\zeta])$ such that $p_*(\tilde{v})=v$. The formula is 
\begin{lemma}\label{lem inf}
 $K_v(\zeta)=\Theta|_{O_{P(E)}(-1)}(\tilde{v},\bar{\tilde{v}})$. 
\end{lemma}

\subsection{Lemmas}
We collect some lemmas in this subsection that will be used later.
\begin{lemma}\label{neg char}
Let $F$ be a strongly pseudoconvex Finsler metric on $E$. The following are equivalent.
\begin{align*}
(1) &\text{ The Finsler metric $F$ has negative Kobayashi curvature}.\\
(2) &\text{ $F$ is strongly plurisubharmonic on $E \setminus \{0\}$}.\\ 
(3) &\text{ $F^2$ is strongly plurisubharmonic on $E \setminus \{0\}$}.
\end{align*}
If any of the above happens, then $F$ is plurisubharmonic on $E$. 
\end{lemma}
\begin{proof}
    For the equivalence between (1) and (2), one can use \cite[(3) and (4) in Theorem 1.3]{DemaillyMSRI} along with Lemma \ref{K positive}. 

    The implication from (2) to (3) is obvious. From (3) to (2), since the Finsler metric $F$ is strongly pseudoconvex, we can use the normal coordinates (\ref{normal}) to compute the complex Hessian of $F$ at one point, which gives
\begin{equation*}
\begin{pmatrix}
    \frac{1}{4}G^{-\frac{3}{2}}(2G_{\alpha\bar{\beta}}G-G_\alpha G_{\bar{\beta}}), & \frac{1}{4}G^{-\frac{3}{2}}(2G_{\alpha\bar{j}}G-G_\alpha G_{\bar{j}})\\
    \frac{1}{4}G^{-\frac{3}{2}}(2G_{i\bar{\beta}}G-G_iG_{\bar{\beta}}),& F_{i\Bar{j}}
\end{pmatrix}=
\begin{pmatrix}
    \frac{1}{4}G^{-\frac{3}{2}}2G_{\alpha\bar{\beta}}G, & 0\\
    0,& F_{i\Bar{j}}
\end{pmatrix}
\end{equation*}    
a positive matrix. So, $F$ is strongly plurisubharmonic. (The equivalence between (1) and (2) can also be proved using Lemma \ref{lem inf} and computing the needed complex Hessian with the normal coordinates (\ref{normal})). 
        
The last statement that $F$ is plurisubharmonic on $E$ is true because $\log F(\phi )$ is plurisubharmonic for any local holomorphic section $\phi$ by Lemma \ref{lem inf} and then we use \cite[Lemma 6.1]{wuwess}. Alternatively, since $F$ is non-negative and is zero on the zero section of $E$, if $F$ is strongly plurisubharmonic on $E \setminus \{0\}$, then $F$ is plurisubharmonic on $E$ for $F$ satisfying the mean value inequality at every point in $E$.
\end{proof}

It is known that a Hermitian metric is Griffiths positive if and only if its dual Hermitian metric is Griffiths negative. The following lemma is a Finsler analogue, and the core of the lemma is due to \cite{Sommese} and \cite[Theorem 2.5]{DemaillyMSRI}. The convexity assumption is important, as we are taking dual Finsler metrics in the proof (see \cite[Remark 2.7]{DemaillyMSRI}). We will call a vector bundle $E$ convex Kobayashi positive (negative) if $E$ carries a convex and strongly pseudoconvex Finsler metric with positive (negative) Kobayashi curvature. In \cite[Lemma 2.1]{LLextrapolation}, Lempert also obtained some duality results for Banach bundles.

\begin{lemma}\label{sommese}
A vector bundle $E$ is convex Kobayashi positive if and only if its dual $E^*$ is convex Kobayashi negative.
\end{lemma}

Before proving Lemma \ref{sommese}, we need another lemma first. This lemma has appeared in Proposition 1.8 in the first version of \cite{FengLiuWan} on the arXiv. We give a simplified proof here. 

\begin{lemma}\label{lem signature}
A Finsler metric $F$ has transversal Levi signature $(r,n)$ if and only if $F$ is strongly pseudoconvex and has positive Kobayashi curvature.
\end{lemma}

\begin{proof}
Recall that a Finsler metric $F$ is said to have transversal Levi signature $(r,n)$, if at every point $(z,\zeta)\in E\setminus \{\text{zero section}\}$, the Levi form $i\partial\bar{\partial}F$ on $E$ is positive definite along the fiber $E_z$ and negative definite on some $n$-dimensional subspace $W\subset T^{1,0}_{(z,\zeta)}E$ which is transversal to the fiber $E_z$.

For the forward direction, since $F_{i\bar{j}}>0$, the Finsler metric $F$ is strongly pseudoconvex. Let $\pi:E\setminus\{\text{zero section}\}\to P(E)$ be the quotient map. As before, we denote $F^2$ by $G$, and the corresponding Hermitian metric on $O_{P(E)}(-1)$ by $h_G$. It is not hard to verify that $\pi^*\Theta(h_G)=-\partial \bar{\partial }\log G$. For any nonzero $e\in W$,
\begin{align*}
 \Theta(h_G)(\pi_*e,\overline{\pi_*e})=\pi^*\Theta(h_G)(e,\overline{e})
 =-2\partial \bar{\partial }\log F(e,\overline{e})\\
 =-2(\frac{\partial \bar{\partial }F}{F}-\frac{\partial F\wedge \bar{\partial}F }{F^2})(e,\overline{e})\geq -2\frac{\partial \bar{\partial }F}{F}(e,\overline{e})>0. 
\end{align*}
As a consequence, the map $\pi_*|_W$ is injective. So the space $\pi_*W$ is of dimension $n$, and $\Theta(h_G)$ is positive definite on $\pi_*W$. By Lemma \ref{fiber neg} and a dimension count, $\Theta(h_G)$ has signature $(n,r-1)$, thus $F$ is Kobayashi positive by Lemma \ref{K positive}.

For the backward direction, we fix a point $(z,\zeta)\in E\setminus \{\text{zero section}\}$. Since $F$ is strongly pseudoconvex, the Levi form $i\partial\bar{\partial}F$ on $E$ is positive definite along the fiber $E_z$. Moreover, we can find a holomorphic frame $\{e_i\}$ for $E$ around $z\in X$ such that 
\begin{equation}\label{normal 1}
  G_{\alpha}(z,\zeta)=G_{\alpha\bar{j}}(z,\zeta)=0,
\end{equation} namely the normal coordinates as in (\ref{normal}). With respect this coordinate system, the tangent space $T^{1,0}_{(z,\zeta)}E$ has a basis $\{ \partial/\partial z_1,\cdots,\partial/\partial z_n, \partial/\partial \zeta_1,\cdots, \partial/\partial \zeta_r  \}$, and the space $E_z=T^{1,0}_\zeta E_z$ has a basis $\{\partial/\partial \zeta_1,\cdots, \partial/\partial \zeta_r  \}$. We choose $W$ to be the space generated by $\{ \partial/\partial z_1,\cdots, \partial/\partial z_n  \}$ which is transversal to $E_z$. 

Let us verify that $i\partial\bar{\partial}F$ is negative definite on $W$. A straightforward computation gives
\begin{align*}
\partial\bar{\partial}F=\partial\bar{\partial} \sqrt{G}=\frac{1}{2} G^{-\frac{1}{2}}\partial\bar{\partial} G   -\frac{1}{4} G^{-\frac{3}{2}}\partial G\wedge \bar{\partial}G. 
\end{align*}
Using (\ref{normal 1}), we get 
\begin{equation*}
    \partial\bar{\partial}F(\frac{\partial}{\partial z_\alpha}, \frac{\partial}{\partial \bar{z}_\beta})= \frac{1}{2} G^{-\frac{1}{2}} G_{\alpha \bar{\beta}}   -\frac{1}{4} G^{-\frac{3}{2}}G_\alpha G_{\bar{\beta}}=\frac{1}{2} G^{-\frac{1}{2}} G_{\alpha \bar{\beta}};
\end{equation*}
the matrix $(G_{\alpha\bar{\beta}})$ in the end differs from the Kobayashi curvature of $F$ by a negative sign (\cite[Formulas (5.10) and (5.16)]{ComplexFinsler}). Since $F$ is Kobayashi positive, the form $i\partial\bar{\partial}F$ is negative definite on $W$, and hence $F$ has transversal Levi signature $(r,n)$.
\end{proof}

\begin{proof}[Proof of Lemma \ref{sommese}]

Assume the bundle $E$ carries a convex and strongly pseudoconvex Finsler metric $F$ with positive Kobayashi curvature. By adding a small Hermitian metric, we may assume the Finsler metric $F$ is strongly convex. By Lemma \ref{lem signature}, the Finsler metric $F$ has transversal Levi signature $(r,n)$.  According to \cite[Theorem 2.5]{DemaillyMSRI}, the dual Finsler metric of $F$ is strongly plurisubharmonic on $E^*\setminus \{0\}$; in particular, it is strongly pseudoconvex. Therefore, the dual Finsler metric of $F$ has negative Kobayashi curvature by Lemma \ref{neg char}. Since the dual metric is already convex, we conclude that the bundle $E^*$ is convex Kobayashi negative.

For the other direction, we assume the dual bundle $E^*$ carries a convex, strongly pseudoconvex Finsler metric $F^*$ whose Kobayashi curvature is negative. By adding a small Hermitian metric, we may assume the Finsler metric $F^*$ is strongly convex. By Lemma \ref{neg char}, the Finsler metric $F^*$ is strongly plurisubharmonic on $E^* \setminus \{0\}$. By \cite[Theorem 2.5]{DemaillyMSRI}, the dual of $F^*$ has transversal Levi signature $(r,n)$. Therefore, by Lemma \ref{lem signature}, the dual of $F^*$ is strongly pseudoconvex and has positive Kobayashi curvature. As a result, the bundle $E$ is convex Kobayashi positive.    
\end{proof}

The next lemma is from \cite[Corollary 2]{wu_2022}, and almost all our arguments start with this lemma. It basically says that when dealing with Kobayashi positive Finsler metrics, we have convexity for free.

\begin{lemma}\label{convex}
If $E$ is Kobayashi positive, then $E$ is convex Kobayashi positive. Therefore, the dual bundle $E^*$ is convex Kobayashi negative by Lemma \ref{sommese}.
\end{lemma}

The next lemma is from  \cite[Theorem 6.1]{Negfinsler}. One should notice that a ``convex'' Finsler metric in \cite{Negfinsler} is what we call ``strongly pseudoconvex'' here. This lemma can also be proved using Lemma \ref{neg char}. 
\begin{lemma}\label{subbundle}
Assume $E$ carries a convex and strongly pseudoconvex Finsler metric $F$ whose Kobayashi curvature is negative. If $E_1$ is a subbundle of $E$, then the restriction of the Finsler metric $F$ to $E_1$ is still convex, strongly pseudoconvex, and Kobayashi negative.    
\end{lemma}

\subsection{Regularization of continuous Finsler metrics}\label{sec regularize}

Recall that in the definition of Finsler metrics, we include the smoothness of $F$ on $E\setminus \{0\}$, so for us Finsler metrics are smooth. If we drop the smoothness assumption, then we call $F$ a continuous Finsler metric. In this subsection, we will show that continuous Finsler metrics with extra properties can be approximated by Finsler metrics with the same properties.

First of all, the correspondence (\ref{corres}) can be extended to the continuous case; namely, there is a one-to-one correspondence between continuous Finsler metrics on $E$ and continuous Hermitian metrics on $O_{P(E)}(-1)$.

We fix a background Finsler metric $F$ on $E$, and denote $F^2$ by $G$. We denote the corresponding Hermitian metric on $O_{P(E)}(-1)$ by $h$, so according to (\ref{corres}) we have $h(z, [\zeta],\zeta)=G(z,\zeta)$. Every continuous Hermitian metric on the line bundle $O_{P(E)}(-1)$ can be written as $h$ times a positive continuous function $m(z,[\zeta])$ on $P(E)$. Hence, through the one-to-one correspondence, every continuous Finsler metric on $E$ can be written as $G(z,\zeta) m(z, [\zeta])$ where $m(z,[\zeta])$ is a positive continuous function on $P(E)$. We will simply write $Gm$ or $Gm (z,\zeta)$ for $G(z,\zeta) m(z, [\zeta])$. (Strictly speaking, we should use $(Gm)^{1/2}$ for Finsler metric, but we abuse the notation by using $Gm$.)

For an open set $U$ in $P(E)$ and a positive continuous function $m(z,[\zeta])$ defined on $U$, the continuous Finsler metric $Gm$ is defined on some open subset of $E$. We call $Gm$ strongly convex if for $\zeta_0$ in $E_{z_0}$, there exist a neighborhood in $E$ and a local Hermitian metric $H$ on $E$ such that $\zeta \mapsto Gm(z,\zeta) -H(\zeta, \zeta)$ is convex in the neighborhood of $(z_0,\zeta_0)$ in $E$. When $m(z,[\zeta])$ is smooth, this definition of strong convexity is consistent with the earlier definition in Subsection \ref{subsec finsler}. On the other hand, $Gm$ is called strongly plurisubharmonic on the total space $E\setminus \{0\}$ if for $\zeta_0\in E_{z_0}$ there exist a coordinate neighborhood in $E$ and a positive number $c$ such that $Gm(z,\zeta)-c(\|z\|^2+\|\zeta\|^2)$ is plurisubharmonic in the neighborhood of $(z_0,\zeta_0)$ in $E$ where $\|z\|^2=\sum_\alpha |z_\alpha|^2$ and $ \|\zeta\|^2=\sum_j|\zeta_j|^2$.

We define a presheaf $S$ on $P(E)$ as follows. For $U$ an open set in $P(E)$, $S(U)$ is set to be the collection of all positive continuous functions $m(z,[\zeta])$ on $U$ such that $Gm(z,\zeta) $ is strongly plurisubharmonic and strongly convex (in the sense of the previous paragraph). It is not hard to see that $S$ satisfies the gluing axioms for sheaves, so $S$ is a sheaf.

We are going to use the following theorem from \cite[Theorem 4.1 and Corollary 1]{GreeneWu}.
\begin{theorem}\label{greenewu}
If $T$ is a subsheaf of the sheaf of germs of continuous functions on $P(E)$ and if $T$ has the local approximation, the $C^{\infty}$ stability, and the maximum closure properties, then the smooth global sections of $T$ are dense in the global sections of $T$ in the $C^0$ fine topology.    
\end{theorem}
Let us explain the terminologies (they can be found in \cite[Definitions 1.1, 1.3, and 1.4]{GreeneWu}). The $C^0$ fine topology on the set of continuous functions on $P(E)$ is the topology generated by the sets $$\{ \text{$u$ continuous functions on $P(E)$} : |f-u|<g \text{ on } P(E) \}$$
where $f,g$ are both continuous on $P(E)$ and $g$ is positive.

The subsheaf $T$ is said to have the local approximation property if each point in $P(E)$ has a neighborhood $U$ such that $T|_U$ has the semilocal approximation property. For the precise definition of the semilocal approximation, see \cite[Definition 1.4]{GreeneWu}; roughly speaking, it means that local sections of $T$ can be approximated by smooth sections of $T$ in a neighborhood of a given compact set.

The subsheaf $T$ is said to have the $C^{\infty}$ stability property if the following holds. Suppose that $U$ is an open set in $P(E)$, $K$ is a compact set in $U$, and $f$ is a real-valued function on $U$ such that the germ $\{f\}_{(z,[\zeta])}$ at $(z,[\zeta])$ is in the stalk $T_{(z,[\zeta])}$ for every point $(z,[\zeta])$ in $U$. Then there exists $\varepsilon>0$ such that every real-valued function $g$ smooth in a neighborhood of $K$ with $d_K(0,g)<\varepsilon$ satisfies $\{f+g\}_{(z,[\zeta])}\in T_{(z,[\zeta])}$ for $(z,[\zeta])$ in $K$. Here $d_K(0,g)$ is 
$$d_K(0,g)=\sup_K |g|+\sum_{i=1}^\infty \frac{1}{2^i}\min (1, \|g\|_{K,i}),$$
where $\|g\|_{K,i}$ is the sup of $i$-th partial derivatives of $g$ computed in some coordinate systems covering $K$ (see the first displayed formula in \cite[page 49]{GreeneWu}).

The subsheaf $T$ is said to have the maximum closure property if for any two germs $\{f\}_{(z,[\zeta])},\{g\}_{(z,[\zeta])}$ in $T_{(z,[\zeta])}$, the germ $\{\max(f,g)\}_{(z,[\zeta])}$ is in $T_{(z,[\zeta])}$.

For our sheaf $S$, we have the following lemma. In other words, this lemma says that a strongly plurisubharmonic, strongly convex, continuous Finsler metric can be approximated by the smooth ones.
\begin{lemma}\label{for sheaf S}
   The smooth global sections of $S$ are dense in the global sections of $S$ in the $C^0$ fine topology.
\end{lemma}

\begin{proof}
Let us verify the three properties in Theorem \ref{greenewu} for $S$. For the maximum closure property, we consider two sections $m_1,m_2$ of $S$ defined in a neighborhood of some point in $P(E)$. The function $ \max(m_1,m_2)  $ is obviously positive and continuous. Also, we have $G\max(m_1,m_2)=\max(Gm_1, Gm_2)$. To show that $\max(Gm_1, Gm_2)$ is strongly convex, we fix $\zeta_0\in E_{z_0}$, then there exist local Hermitian metrics $H_1$ and $H_2$ on $E$ such that $Gm_1-H_1$ and $Gm_2-H_2$ are convex in a neighborhood of $(z_0,\zeta_0)$ in $E$. By choosing a small positive number $c$, we can make $Gm_1-cH_1$ and $Gm_2-cH_1$ convex in a neighborhood of $(z_0,\zeta_0)$ in $E$. So $\max(Gm_1-cH_1, Gm_2-cH_1)=\max(Gm_1,Gm_2)-cH_1$ is convex in a neighborhood of $(z_0,\zeta_0)$ in $E$. Hence, $\max(Gm_1, Gm_2)$ is strongly convex.  Showing that $\max(Gm_1, Gm_2)$ is strongly plurisubharmonic is similar, and we skip the details. As a result, the function $ \max(m_1,m_2)  $ is a section of $S$, and the sheaf $S$ has the maximum closure property. 

Since strong plurisubharmonicity and strong convexity are open conditions, the sheaf $S$ has the $C^\infty$ stability property. To be precise, let $U$ be an open set in $P(E)$, $K$ a compact set in $U$, and $m$ a section of $S$ over $U$. So, $m$ is a positive continuous function on $U$ such that $Gm$ is strongly plurisubharmonic and strongly convex. Let $H$ be a Hermitian metric on $E$ and $\widetilde{K}$ be the compact set $\{(z,\zeta)\in E : 1 \leq H(\zeta,\zeta) \leq 2  \}$. By strong convexity of $Gm$, there exist a finite open cover $\{V_j\}$ for $\widetilde{K}$ and finitely many local Hermitian metrics $H_j$ such that $Gm-H_j$ is convex in $V_j$ for each $j$. We can find $\varepsilon>0$ small such that for $g$ smooth in a neighborhood of $K$ with $d_K(0,g)<\varepsilon$,  $Gg+H_j/2$ is convex in $V_j$ for each $j$. Therefore, $$G(m+g)-H_j/2=Gm-H_j+Gg+H_j/2$$
is convex in $V_j$ for each $j$. Now for a fixed $(z,\zeta)\in E$ with $(z,[\zeta]) \in K$, we can find $\lambda \in \mathbb{C} $ to make $(z,\lambda\zeta)$ lie in the interior of $\widetilde{K}$ and hence in $V_j$ for some $j$; since $$G(m+g)(z,\lambda \zeta)-\frac{H_j}{2}(\lambda \zeta, \lambda \zeta)=|\lambda|^2 \big( G(m+g)(z,\zeta)-\frac{H_j}{2}(\zeta, \zeta) \big),$$ $G(m+g)-H_j/2$ is convex in a neighborhood of $(z,\zeta)$ in $E$. We have thus obtained the strong convexity for $G(m+g)$. For the strong plurisubharmonicity, the argument is similar, and we skip the details. 

For the local approximation property, we fix a point $(z_0,[\zeta_0])$ in $P(E)$ and consider a coordinate neighborhood; we want to verify that $S$ restricted to this coordinate neighborhood has the semilocal approximation property. 

We assume the coordinate neighborhood has the following coordinate system. Let $\{e_1,...,e_r\}$ be a holomorphic frame of $E$ around $z_0$ with fiber coordinates $(\zeta_1,...,\zeta_r)$. Around the point $(z_0,[\zeta_0])\in P(E)$, we define the local coordinates $(z_1,\ldots, z_n, w_1,\ldots, w_{r-1})$ by $w_i=\zeta_i/\zeta_r$ for $i=1\sim r-1$. So $$e:=\frac{\zeta_1e_1+\cdots+\zeta_re_r}{\zeta_r}=w_1e_1+\cdots+w_{r-1}e_{r-1}+e_r$$
is a holomorphic frame for $O_{P(E)}(-1)$.

Let $U$ be an open set in this coordinate neighborhood, and $m$ be a section of $S$ over $U$; namely, $m$ is a positive continuous function on $U$ such that $Gm$ is strongly plurisubharmonic and strongly convex. Especially, the continuous Hermitian metric $h(z,[\zeta],\zeta)m(z,[\zeta])$ is plurisubharmonic on the total space $O_{P(E)}(-1)$, so $$\log\big( h(z,[\zeta],e)m(z,[\zeta])\big)$$ is plurisubharmonic on $U$. We will abbreviate $h(z,[\zeta],e)m(z,[\zeta])$ as $(h(e)m)(z,w)$ or $h(e)m $. 

Let $\rho$ be a nonnegative radial smooth function with support in the unit ball that has integral one in $\mathbb{C}^{n+r-1}$. For $\delta>0$, we consider $\rho_\delta(\cdot)=\delta^{-2(n+r-1)}\rho(\cdot/\delta)$, and the convolution 
\begin{equation}\label{convolution}
  \big(h(e)m\big)*\rho_\delta.  
\end{equation}
The function $$m_\delta:=\frac{\big(h(e)m\big)*\rho_\delta}{ h(e)}$$ is positive and smooth. We claim that $Gm_\delta$ is plurisubharmonic. First, each term in the Riemann sum for the integral in the convolution $\big(h(e)m\big)*\rho_\delta$ is of the form $(h(e)m )(z-s, w-t )\rho_\delta(s,t)\Delta s \Delta t $ where $s,t$ are the sample points and $\Delta s \Delta t$ is the volume of a subinterval in the partition, and these terms are log plurisubharmonic as the function $(h(e)m)(z,w)$ is log plurisubharmonic. The logarithm of the Riemann sum which we denote by 
$$\log \big( \sum  (h(e)m) (z-s, w-t )\rho_\delta(s,t)\Delta s \Delta t   \big)=\log \big( \sum  e^{\log (h(e)m) (z-s, w-t )\rho_\delta(s,t)\Delta s \Delta t}   \big)$$ is therefore plurisubharmonic, and it converges to $\log   \big[(h(e)m)*\rho_\delta\big]=\log (h(e)m_\delta)$ which is plurisubharmonic as a result. Meanwhile, we have $$\log (Gm_\delta)(z, \zeta) = \log\big( |\zeta_r|^2  (h(e)m_\delta)(z,w)\big)=\log |\zeta_r|^2 +\log (h(e)m_\delta)(z,w), $$ where the last two terms are plurisubharmonic, so $Gm_\delta$ is plurisubharmonic as we claimed. 

Next we claim that, for $\zeta$ and $ \xi $ in $E_z$ and $0\leq \theta\leq 1$,  $Gm_\delta (z, \theta\zeta+(1-\theta)\xi)\leq \theta Gm_\delta (z, \zeta)+(1-\theta)Gm_\delta (z, \xi)$. Let us unwind the convolution first,
\begin{align*}
    &Gm_\delta(z, \zeta)= |\zeta_r|^2  (h(e)m_\delta) (z,w)\\
    =&|\zeta_r|^2 \int_{(s,t)\in \mathbb{C}^n\times\mathbb{C}^{r-1}} (h(e)m)(z-s,w-t) \rho_\delta(s,t)dsdt\\
    =&\int_{(s,t)\in \mathbb{C}^n\times\mathbb{C}^{r-1}}  h\big(z-s, \big[\zeta-\zeta_r\sum_{i=1}^{r-1}t_ie_i\big] ,\zeta-\zeta_r\sum_{i=1}^{r-1}t_ie_i\big)  m(z-s,w-t) \rho_\delta(s,t)dsdt\\
    =&\int_{(s,t)\in \mathbb{C}^n\times\mathbb{C}^{r-1}}  G\big(z-s, \zeta-\zeta_r\sum_{i=1}^{r-1}t_ie_i\big)  m(z-s,w-t) \rho_\delta(s,t)dsdt\\
    =&\int_{(s,t)\in \mathbb{C}^n\times\mathbb{C}^{r-1}}  Gm\big(z-s, \zeta-\zeta_r\sum_{i=1}^{r-1}t_ie_i\big)\rho_\delta(s,t)dsdt.
    \end{align*}
where $w=(w_1,...,w_{r-1})$ and $\zeta_i/\zeta_r=w_i$ for $i=1\sim r-1$, and $\{e_i\}$ is the holomorphic frame for $E$.  
Therefore, 
\begin{align*}
    &Gm_\delta (z, \theta\zeta+(1-\theta)\xi)\\
    =&\int_{(s,t)\in \mathbb{C}^n\times\mathbb{C}^{r-1}}  Gm\big(z-s, \theta\zeta+(1-\theta)\xi-(\theta\zeta_r+(1-\theta)\xi_r)\sum_{i=1}^{r-1}t_ie_i\big) \rho_\delta(s,t)dsdt\\
    =&\int_{(s,t)\in \mathbb{C}^n\times\mathbb{C}^{r-1}}  Gm\big(z-s,    
    \theta(\zeta -\zeta_r\sum_{i=1}^{r-1}t_ie_i   )       +(1-\theta)( \xi-\xi_r\sum_{i=1}^{r-1}t_ie_i)\big) \rho_\delta(s,t)dsdt\\    
    \leq&  \int_{(s,t)\in \mathbb{C}^n\times\mathbb{C}^{r-1}}  \theta Gm\big(z-s,    
    \zeta -\zeta_r\sum_{i=1}^{r-1}t_ie_i \big)+
    (1-\theta)Gm\big(z-s,    
     \xi-\xi_r\sum_{i=1}^{r-1}t_ie_i\big) \rho_\delta(s,t)dsdt\\
     =&\theta Gm_\delta (z, \zeta)+(1-\theta)Gm_\delta (z, \xi),
\end{align*}
where the inequality is due to the convexity of $Gm$. 

With the local frame $\{e_i\}$ for $E$, we can construct a local Hermitian metric $H$ on $E$ which is Griffiths negative (for example, $H(z,\zeta)=\|\zeta\|^2e^{\|z\|^2}$). So the Finsler metric $Gm_\delta+\delta H$ is strongly plurisubharmonic and strongly convex. The function $$\frac{Gm_\delta+\delta H}{G}(z,\zeta)$$     
is therefore a smooth section of the sheaf $S$, and it approximates $m$ in the sense of the conditions (a) and (b) in \cite[Definition 1.4]{GreeneWu} for $\delta$ small (with $f$ and $ F$ there replaced with $m$ and $ (Gm_\delta+\delta H)/G$ respectively). 

As a result, $S$ restricted to the coordinate neighborhood has the semilocal approximation property, and hence $S$ has the local approximation property. (One could consider the convolution $\big(\log (h(e)m)\big)*\rho_\delta$ instead of (\ref{convolution}) but we do not know how to prove convexity in this case. Also, taking the convolution of $Gm$ on the total space $E$ does not seem to work because the homogeneity of Finsler metrics is lost).

\end{proof}

There is another sheaf we will use later. Define a presheaf $S'$ on $P(E)$ as follows. For $U$ an open set in $P(E)$, $S'(U)$ is set to be the collection of all positive continuous functions $m(z,[\zeta])$ on $U$ such that $Gm(z,\zeta)$ is strongly convex (in the sense of continuous Finsler metrics); moreover, for each point in $X$, there exist a coordinate neighborhood and $\varepsilon>0$ such that for any local holomorphic section $A$ of $E$ in this coordinate system, the function 
\begin{equation}\label{AA}
  \log \big(Gm(z,A(z))\big)-\varepsilon\|z\|^2  
\end{equation}
is plurisubharmonic where $\|z\|^2=\sum_\alpha |z_\alpha|^2$ (we will simply write $\log Gm(A(z))$ later). Again, it is not hard to see that $S'$ is a sheaf.

\begin{lemma}\label{for sheaf S'}
   The smooth global sections of $S'$ are dense in the global sections of $S'$ in the $C^0$ fine topology.
\end{lemma}
\begin{proof}
The proof is similar to that of Lemma \ref{for sheaf S}. Let us verify the three properties in Theorem \ref{greenewu} for $S'$. For the maximum closure property, we consider two sections $m_1,m_2$ of $S'$ defined in a neighborhood of some point in $P(E)$. The function $ \max(m_1,m_2) $ is obviously positive and continuous. Also, we have $G\max(m_1,m_2)=\max(Gm_1, Gm_2)$. That $\max(Gm_1, Gm_2)$ is strongly convex is proved exactly the same as in the proof of Lemma \ref{for sheaf S}. Next, given a point in $X$, we can find a coordinate neighborhood and $\varepsilon>0$ such that $\log Gm_1(A(z))-\varepsilon\|z\|^2$ and $\log Gm_2(A(z))-\varepsilon\|z\|^2$ are plurisubharmonic for any local holomorphic section $A$ of $E$. So, $\log\max(Gm_1, Gm_2 )(A(z))-\varepsilon\|z\|^2$ is plurisubharmonic for any $A$. Consequently, $ \max(m_1,m_2) $ is a section of $S'$, and the sheaf $S'$ has the maximum closure property. 

Since the requirements in defining the sheaf $S'$ are open conditions, the sheaf $S'$ has the $C^\infty$ stability property. Indeed, let $U$ be an open set in $P(E)$, $K$ a compact set in $U$, and $m$ a section of $S'$ over $U$. Let $K'$ be the image of $K$ under the projection $P(E)\to X$. Since $K'$ is compact, there exist a finite open cover $\{V_j\}$ for $K'$ and $\varepsilon_{\min}>0$ such that the condition (\ref{AA}) holds in each $V_j$ with $\varepsilon_{\min}$. We can find $\varepsilon'>0$ such that for $g$ smooth in a neighborhood of $K$ with $d_K(0,g)< \varepsilon'$, $$\log(1+\frac{g}{m}(z,[A(z)]))+\frac{\varepsilon_{\min}}{2}\|z\|^2$$ is plurisubharmonic for any holomorphic section $A$ of $E$ defined in $V_j$ for each $j$ ($(z,[A(z)])$ has to be in the domain of $g$). Therefore, $$\log\big(  G(m+g)(A(z))  \big)-\frac{\varepsilon_{\min}}{2}\|z\|^2=\log  Gm(A(z)) -\varepsilon_{\min}\|z\|^2+\log \big(  1+\frac{g}{m}([A(z)]) \big) +\frac{\varepsilon_{\min}}{2}\|z\|^2  $$
is plurisubharmonic for any holomorphic section $A$ of $E$ defined in $V_j$ for each $j$. We have verified the condition (\ref{AA}). The strong convexity is verified the same way as in the proof of Lemma \ref{for sheaf S}.

For the local approximation property, we fix a point $(z_0,[\zeta_0])$ in $P(E)$ and consider a coordinate neighborhood; we want to verify that $S'$ restricted to this coordinate neighborhood has the semilocal approximation property. 

As before, this coordinate neighborhood can be described as follows.  Let $\{e_1,...,e_r\}$ be a holomorphic frame of $E$ around $z_0$ with fiber coordinates $(\zeta_1,...,\zeta_r)$. Around the point $(z_0,[\zeta_0])\in P(E)$, we define the local coordinates $(z_1,\ldots, z_n, w_1,\ldots, w_{r-1})$ by $w_i=\zeta_i/\zeta_r$ for $i=1\sim r-1$. So $$e:=\frac{\zeta_1e_1+\cdots+\zeta_re_r}{\zeta_r}=w_1e_1+\cdots+w_{r-1}e_{r-1}+e_r$$
is a holomorphic frame for $O_{P(E)}(-1)$.

Let $U$ be an open set in this coordinate neighborhood, and $m$ be a section of $S'$ over $U$. The condition (\ref{AA}) implies that $Gm(z,\zeta)$ is plurisubharmonic (\cite[Lemma 6.1]{wuwess}). In the proof of the local approximation for Lemma \ref{for sheaf S}, we only use the fact $Gm$ is plurisubharmonic and strongly convex. As the $Gm$ in the present case is plurisubharmonic and strongly convex, we conclude that the Finsler metric $Gm_\delta+\delta H$, constructed as in the proof of Lemma \ref{for sheaf S}, is strongly plurisubharmonic and strongly convex. Since we have smoothness for the Finsler metric $Gm_\delta+\delta H$, it is Kobayashi negative by Lemma \ref{neg char}, and so it satisfies the condition (\ref{AA}) by Lemma \ref{lem inf}. The function $$\frac{Gm_\delta+\delta H}{G}(z,\zeta)$$     
is therefore a smooth section of the sheaf $S'$, and it approximates $m$ in the sense of the conditions (a) and (b) in \cite[Definition 1.4]{GreeneWu} for $\delta$ small (with $f$ and $ F$ there replaced with $m$ and $ (Gm_\delta+\delta H)/G$ respectively). 

As a result, $S'$ restricted to the coordinate neighborhood has the semilocal approximation property, and hence $S'$ has the local approximation property.   
\end{proof}

If one does not need convexity and only wants to show that a strongly plurisubharmonic continuous Finsler metric can be approximated by the smooth ones, then a much shorter proof can be given based on the argument above. One can also use Richberg's regularization theorem \cite{Richberg} in this case.

\section{Proof of main results}
Let us recall the statements of Theorem \ref{thm 1}:
\begin{enumerate}
    \item\label{11} The quotient bundle of a Kobayashi positive vector bundle is Kobayashi positive.    
    \item\label{12} If $E_1$ and $E_2$ are Kobayashi positive, then $E_1\otimes E_2$ is Kobayashi positive.
    \item If $E$ is Kobayashi positive, then the tensor power $E^{\otimes k}$ and the symmetric power $S^kE$ are Kobayashi positive for $k\geq 1$, and the exterior power $\bigwedge^k E$ is Kobayashi positive for $1\leq k \leq \text{rank } E  $.
\end{enumerate}
\begin{proof}[Proof of Theorem \ref{thm 1}]
\

1. Let $E$ be a Kobayashi positive vector bundle and $E_2$ be a quotient bundle of $E$. By Lemma \ref{convex}, the dual bundle $E^*$ is convex Kobayashi negative. Since $E_2^*$ is a subbundle of $E^*$, the bundle $E^*_2$ is convex Kobayashi negative by Lemma \ref{subbundle}. Using Lemma \ref{sommese}, the bundle $E_2$ is convex Kobayashi positive.

One can give another proof based on \cite[Lemma 2.3]{LLextrapolation} who uses the quotient norm. However, before applying Lempert's lemma, one has to make sure the metrics are convex, and this is justified since Kobayashi positivity implies convexity, Lemma \ref{convex}.

2.
By Lemma \ref{convex}, we assume that $E_1$ carries a convex and strongly pseudoconvex Kobayashi positive Finsler metric $F_1$, and $E_2^*$ carries a convex and strongly pseudoconvex Kobayashi negative Finsler metric $F_2$. We have the operator norm on the bundle $\Hom (E_1,E_2^*)$ defined by 
\begin{equation*}
     F_3(z, A) := \sup_{0\neq \zeta\in E_1|_z} \frac{F_2(z,A\zeta)}{F_1(z,\zeta)} \text{\,\,  for $A\in \Hom (E_1,E_2^*)|_z$ }.
\end{equation*}
Obviously, $F_3$ satisfies positivity and homogeneity in the definition of a Finsler metric. Using the continuity of $F_1$ and $F_2$, it is not hard to verify that $F_3$ is continuous on the total space $\Hom (E_1,E_2^*)$, but we do not know if $F_3$ is smooth away from the zero section. So $F_3$ is only a continuous Finsler metric which is sufficient for our proof later. In addition, the continuous Finsler metric $F_3$ satisfies $F_3(z, A+B )\leq F_3(z,A)+F_3(z,B) $ since the Finsler metric $F_2$ is convex. Later on, if the context is clear, we will simply write
$F_3(z,A)=F_3(A)$. 


Let $H$ be a Hermitian metric on $\Hom (E_1, E_2^*)$. Fix a point in $X$ and a coordinate neighborhood $U$ with coordinates $(z_1,\ldots,z_n)$. Since the Finsler metric $F_2$ is Kobayashi negative, its Kobayashi curvature $\Theta|_{O_{P(E^*_2)}(-1)}$ satisfies
\begin{equation*}
\Theta|_{O_{P(E^*_2)}(-1)} \leq -2\varepsilon \sum_{\alpha} dz_\alpha \wedge d\bar{z}_{\alpha}
\end{equation*}
in $P(E_2^*|_U)$ for some $\varepsilon>0$ by using the local formula (\ref{local for koba}). Together with Lemma \ref{lem inf}, we get 
\begin{equation}\label{2 epsilon}
   \partial\bar{\partial}\log F^2_2(\phi)(v,\bar{v})\geq 2\varepsilon \sum_{\alpha} |v_\alpha|^2, 
\end{equation}
where $\phi$ is any local holomorphic section of $E_2^*$ and $v=\sum v_\alpha \partial/\partial z_\alpha$ is any vector field in $U$. In particular, there exists a positive constant $c_0$ such that for $0\leq c\leq c_0$, we have
\begin{equation}\label{2 epsilon'}
   \partial\bar{\partial}\log \big[F^2_2(A(z)\psi(z))+cH(A(z)) F^2_1(\psi(z))\big]  (v,\bar{v})\geq \varepsilon \sum_{\alpha} |v_\alpha|^2, 
\end{equation}
where $A(z)$ is any local holomorphic section of $\Hom (E_1,E_2^*)$ and  $\psi(z)$ is any local holomorphic section of $E_1$, and $v=\sum v_\alpha \partial/\partial z_\alpha$ is any vector field in $U$.

Now, we fix $c$ with $0\leq c\leq c_0$ and a local holomorphic section $A(z)$ of $\Hom (E_1,E_2^*)$ defined in $U$. From (\ref{2 epsilon'}), for any local holomorphic section $\psi$ of $E_1$, the function 
\begin{equation}\label{apsi}
  \log \big[F^2_2(A\psi) +cH(A )F^2_1(\psi)\big]   -\varepsilon \|z\|^2  
\end{equation}
is plurisubharmonic, where $\|z\|^2=\sum_\alpha|z_\alpha|^2$. Next, we let 
\begin{align*}
    h(z):&=\log \big[F^2_3(A(z))+cH(A(z))\big]-\varepsilon \|z\|^2\\
    &=\log \sup_{0\neq \zeta\in E_1|_z} \frac{\big[F^2_2(A(z) \zeta)+cH(A(z))F_1^2(\zeta)     \big] e^{-\varepsilon \|z\|^2}}{F^2_1(\zeta)}. 
\end{align*}
We are going to show that the function $h(z)$ is plurisubharmonic in $U$. Fix a complex line in $U$. Without loss of generality, we assume the complex line has the direction in $z_1$, namely, $\mathbb{C}\ni \lambda\mapsto p_0+\lambda (1,0,\cdots,0) $ where $p_0$ is a point in $U$. For a point $z_0$ on the complex line, we can find a nonzero vector $ \zeta_0\in E_1|_{z_0}$ such that
\begin{equation}\label{equality}
    h(z_0)=\log  \frac{\big[F^2_2(A(z_0) \zeta_0)+cH(A(z_0))F_1^2(\zeta_0)     \big]e^{-\varepsilon \|z_0\|^2}}{F^2_1(\zeta_0)}. 
\end{equation}

Since the Finsler metric $F_1$ is Kobayashi positive, if we choose $v=v_1\partial/\partial z_1 \in T^{1,0}_{z_0}X$ in Lemma \ref{lem inf}, then
\begin{equation}
0<-\inf \partial\bar{\partial}\log F_1^2(\psi)\bigr|_{z_0}(v,\bar{v}),
\end{equation}
the inf taken over local holomorphic sections $\psi$ of $E_1$ such that $\psi(z_0)=\zeta_0$. We can find such a $\psi$ so that  
\begin{equation}\label{delta}
0>\partial\bar{\partial}\log F_1^2(\psi)\bigr|_{z_0}(v,\bar{v})= \frac{\partial^2\log F_1^2(\psi)}{\partial z_1 \partial \bar{z}_1}\Bigr|_{z_0}|v_1|^2.
\end{equation}
By the definition of $h(z)$, we have 
\begin{equation}\label{3.4}
h(z)\geq \log  \frac{\big[F^2_2(A(z) \psi(z))+cH(A(z))F_1^2(\psi(z))     \big]e^{-\varepsilon \|z\|^2}}{F^2_1(\psi(z))} \text{\,\, for $z$ near $z_0$ in $U$}   
\end{equation}
with equality at $z_0$ by (\ref{equality}). Now, we restrict to the complex line and denote by $\fint_{B(z_0,r)}h(z)$ the average of $h$ over a ball centered at $z_0$ with radius $r$ in the complex line. The inequality (\ref{3.4}) yields 
\begin{align*}
&\big(\fint_{B(z_0,r)} h(z)\big) -h(z_0) \\
\geq& \big(\fint_{B(z_0,r)} \log \big[F^2_2(A\psi) +cH(A )F^2_1(\psi)\big]e^{-\varepsilon \|z\|^2}-\log F^2_1(\psi)\big)\\ &- \big(\log  \big[F^2_2(A(z_0)\zeta_0) +cH(A(z_0) )F^2_1(\zeta_0)\big] e^{-\varepsilon \|z_0\|^2}- \log F^2_1(\zeta_0)\big)\\
\geq&\big(\fint_{B(z_0,r)} -\log F^2_1(\psi)\big) + \log F^2_1(\zeta_0),
\end{align*}
where the second inequality is due to the fact $  \log \big[F^2_2(A\psi) +cH(A )F^2_1(\psi)\big]   -\varepsilon \|z\|^2 $ is plurisubharmonic, namely (\ref{apsi}). As a result, 
\begin{align*}
    &\liminf_{r\to 0} \frac{1}{r^2} \big(\fint_{B(z_0,r)} h(z)-h(z_0)\big)\geq \liminf_{r\to 0} -\frac{1}{r^2}  \big(\fint_{B(z_0,r)} \log F^2_1(\psi) - \log F^2_1(\zeta_0)\big)\\
    =&-\frac{ \partial^2\log F^2_1(\psi)}{\partial z_1 \partial \Bar{z}_1 }\Bigr|_{z_0}>0,
\end{align*}
where we use (\ref{delta}) in the last inequality. Hence $\liminf_{r\to 0} \frac{1}{r^2} \big(\fint_{B(z_0,r)} h(z)-h(z_0)\big)\geq 0$, and $h$ is subharmonic on the complex line (for the last argument, one can see \cite{Saks}, \cite[Lemma 11.2]{CofimanSemmes}, \cite[Lemma 2.3]{LLmax}, or \cite[Lemma 2.2]{LLextrapolation}).

In summary, we have shown that given a point in $X$ and a coordinate neighborhood $U$, there exist $\varepsilon>0$ and a positive constant $c_0$ such that for $0\leq c\leq c_0$ and for any holomorphic section $A(z)$ of $\Hom (E_1, E_2^*)$ defined in $U$, the function $\log [F^2_3(A(z))+cH(A(z))]-\varepsilon \|z\|^2$ is plurisubharmonic.

Since $X$ is compact, we can find $\varepsilon_{\min}>0$ and $c_{\min}>0$ such that for each point in $X$ there exists a coordinate neighborhood so that 
the function $\log (F^2_3+c_{\min}H)(A(z))-\varepsilon_{\min} \|z\|^2$ is plurisubharmonic for any local holomorphic section $A(z)$ of $\Hom (E_1, E_2^*)$ defined in this neighborhood.

As a result, the continuous Finsler metric $(F_3^2+c_{\min} H)$ on the bundle $\Hom (E_1, E_2^*)$ satisfies the condition (\ref{AA}) and is strongly convex. By Lemma \ref{for sheaf S'}, there exists a Finsler metric $F_4$ on $\Hom (E_1, E_2^*)$ that is strongly convex and satisfies the condition (\ref{AA}). By Lemma \ref{lem inf}, the Finsler metric $F_4$ is Kobayashi negative. As a consequence, the bundle $\Hom(E_1, E_2^*)\simeq E_1^*\otimes E_2^*$ is convex Kobayashi negative, hence  $E_1\otimes E_2$ is convex Kobayashi positive by Lemma \ref{sommese}.



In \cite[Corollary 2.5]{LLmax}, Lempert formulated a theorem which is similar to ours. But he assumed that the metric on $E_1$ is Hermitian and the metric on $E_2^*$ is convex, which are stronger than our assumption.

3.
It is an immediate consequence of the statements \ref{11} and \ref{12}.
\end{proof}

Recall the statements of Theorem \ref{thm 2}: \begin{enumerate}
    \item The direct sum of vector bundles $E_1\oplus E_2$ is Kobayashi positive if and only if $E_1$ and $E_2$ are both Kobayashi positive.
    \item In the short exact sequence of vector bundles $0\to E_1\to E \to E_2\to 0$, if $E_1$ and $E_2$ are Kobayashi positive, then $E$ is Kobayashi positive.
    \item If $f:Y\to X$ is an immersion, and $E$ is a Kobayashi positive vector bundle over $X$, then the pull-back $f^*E$ is Kobayashi positive.
\end{enumerate}

\begin{proof}[Proof of Theorem \ref{thm 2}]
\

1. The forward direction follows from the statement \ref{11} in Theorem \ref{thm 1}. For the backward direction, we have two Kobayashi positive vector bundles $E_1$ and $E_2$, hence by Lemma \ref{convex} the duals $E_1^*$ and $E_2^*$ are convex Kobayashi negative, and we denote by $F_1$ and $F_2$ the convex, strongly pseudoconvex, Kobayashi negative Finsler metrics on $E_1^*$ and $E_2^*$ respectively. Let $H_1$ and $H_2$ be two Hermitian metrics on $E_1^*$ and $E_2^*$ respectively. Choose $\varepsilon>0$ small such that the Finsler metrics $F_1^2+\varepsilon H_1$ and $F_2^2+\varepsilon H_2$ are still Kobayashi negative. 

On the direct sum $E^*_1\oplus E^*_2$, we define $F$ by 
\begin{equation}\label{sum}
  F(z,(\xi, \eta))=\big((F_1^2+\varepsilon H_1)(z,\xi)+(F_2^2+\varepsilon H_2)(z,\eta)\big)^{\frac{1}{2}}  
\end{equation}
where $\xi$ and $\eta$ are vectors in $E_1^*|_z$ and $E_2^*|_z$ respectively. It is obvious that $F$ satisfies positivity and homogeneity in the definition of a Finsler metric. However, $F$ is smooth only on $(E^*_1\setminus\{0\})\oplus (E^*_2\setminus\{0\})$ rather than on $E^*_1\oplus E^*_2 \setminus \{0\}$, but $F$ is still continuous on the total space $E^*_1\oplus E^*_2$ since $F_1^2+\varepsilon H_1$ and $F_2^2+\varepsilon H_2$ are continuous on $E^*_1$ and $E^*_2$ respectively. Hence $F$ is only a continuous Finsler metric. By the convexity of $F_1^2+\varepsilon H_1$ and $F_2^2+\varepsilon H_2$, we see that $F^2$ is convex.

By Lemma \ref{neg char}, the Finsler metric $F_1^2+\varepsilon H_1$ is strongly plurisubharmonic on $E_1^*\setminus\{0\}$ and plurisubharmonic on $E_1^*$, and the same is true for $F_2^2+\varepsilon H_2$ on $E_2^*$. Therefore, the continuous Finsler metric $F^2$ is strongly plurisubharmonic on $E^*_1\oplus E^*_2 \setminus \{0\}$.

Let $H$ be a Hermtian metric on $E^*_1\oplus E^*_2$. By choosing $c>0$ small, the continuous Finsler metric $F^2+cH$ is strongly plurisubharmonic and strongly convex. By Lemma \ref{for sheaf S}, there exists a Finsler metric $F_3$ that is strongly plurisubharmonic and strongly convex. Hence $F_3$ is Kobayashi negative by Lemma \ref{neg char}. Therefore, the bundle $E^*_1\oplus E^*_2$ with the Finsler metric $F_3$ is convex Kobayashi negative, so the bundle $E_1\oplus E_2$ is convex Kobayashi positive by Lemma \ref{sommese}.




2. 
This is practically the content of \cite[Lemma 2.2]{Umemura}, but Umemura dealt with Hermitian metrics only. We will follow closely Umemura's argument. We consider the sequence of dual bundles $0\to E_2^*\to E^* \to E_1^*\to 0$. Since $E_1$ and $E_2$ are Kobayashi positive, the duals $E_1^*$ and $E_2^*$ are convex Kobayashi negative. We borrow from the proof of the statement 1 the Finsler metrics $F_1^2+\varepsilon H_1$ and $F_2^2+\varepsilon H_2$, and we denote them by $G_1$ and $G_2$.

Let $E^*$ be determined by $v\in H^1(X, \Hom (E^*_1,E^*_2))$. Take a sufficiently fine open covering $\{U_a\}$ for $X$ so that the extension $E^*$ is given by patching $E^*_2|_{U_a}\oplus E^*_1|_{U_a}$ and $E^*_2|_{U_b}\oplus E^*_1|_{U_b}$ on $U_a \bigcap U_b $ by
\begin{equation*}
\begin{pmatrix}
  I & A_{ba}\\
  0 & I
\end{pmatrix}
\end{equation*}
 where $A_{ba}\in H^0(U_a \bigcap U_b ,\Hom (E^*_1,E^*_2) )$. Then the extension $E^*_\lambda$ determined by $\lambda v\in H^1(X, \Hom (E^*_1,E^*_2))$ with $\lambda$ a nonzero complex number is given by patching through $\big(\begin{smallmatrix}
  I & \lambda A_{ba}\\
  0 & I
\end{smallmatrix}\big)$. Since an extension is differentiably trivial, there exists a smooth homomorphism $B_a: E^*_1|_{U_a} \mapsto E^*_2|_{U_a}$ for each $a$ such that $A_{ba}=B_a-B_b$ on $U_a\bigcap U_b$. In other words, the collection of maps 
\begin{equation*}
\begin{pmatrix}
  I & -\lambda B_{a}\\
  0 & I
\end{pmatrix}: E^*_2|_{U_a}\oplus E^*_1|_{U_a} \mapsto E^*_2|_{U_a}\oplus E^*_1|_{U_a}
\end{equation*}
gives rise to a smooth bundle isomorphism $\Phi$ from $E^*_2\oplus E^*_1$ to $E^*_\lambda$.

On the direct sum $E^*_2\oplus E^*_1$, we define as in (\ref{sum}) a continuous Finsler metric  
$G(z,(\eta,\xi))=G_2(z,\eta)+G_1(z,\xi)$ where $\eta$ and $\xi$ are vectors in $E_2^*|_z$ and $E_1^*|_z$ respectively (we use the convention $G=F^2$). Using the smooth bundle isomorphism $\Phi$, we can define a continuous Finsler metric $\Tilde{G}$ on the bundle $E^*_\lambda$; over $U_a$   
$$\Tilde{G}(z, (\eta, \xi))=G(z, (\eta+\lambda B_a(z)\xi, \xi)). $$
We already know from the proof of the statement 1 that $G$ is convex, so $\Tilde{G}$ is convex. Since $G$ is strongly plurisubharmonic on $E^*_2\oplus E^*_1\setminus \{0\}$, by choosing $\lambda$ small we can make $\Tilde{G}$ strongly plurisubharmonic on $E^*_\lambda \setminus \{0\}$. 

As before, the continuous Finsler metric  $\Tilde{G}+c H$, with $H$ a Hermitian metric on $E^*_\lambda$ and $c>0$ small, is strongly plurisubharmonic and strongly convex. By Lemma \ref{for sheaf S}, there exists a Finsler metric on $E^*_\lambda$ that is strongly plurisubharmonic and strongly convex, and so Kobayashi negative.

Therefore, the bundle $E^*_\lambda$ is convex Kobayashi negative, hence $E_ \lambda\simeq E$ is convex Kobayashi positive by \cite[Lemma 2.1]{Umemura}.

3. We first consider a more general case where the map $f:X\to Y$ is smooth and the bundle $E$ carries a Finsler metric $F$. The pull-back bundle $f^*E$ has the pull-back Finsler metric $f^*F$ defined as  $$f^*F(x,\zeta)=F(f(x),\zeta) \text{, for $\zeta\in E_{f(x)} $.}  $$
It is not hard to check that if $F$ is convex or strongly pseudoconvex, then so is $f^*F$. 

Now, we consider the case in the statement \ref{3}. We assume that $f$ is an immersion and $E$ is Kobayashi positive. By Lemma \ref{convex}, the dual $E^*$ is convex Kobayashi negative, and we let $F^*$ be a convex, strongly pseudoconvex, Kobayashi negative Finsler metric on $E^*$. From the discussion in the previous paragraph, the pull-back Finsler metric $f^*F^*$ on $f^*E^*$ is convex and strongly pseudoconvex. By Lemma \ref{neg char}, the Finsler metric $F^*$ is strongly plurisubharmonic on $E^*-\{0\}$. Since $f$ is an immersion and $f^*F^*(x,\zeta)=F^*(f(x),\zeta) $, a direct computation shows that $f^*F^*$ is strongly plurisubharmonic on $f^*E^*-\{0\}$. By Lemma \ref{neg char} again, the Finsler metric $f^*F^*$ is Kobayashi negative. The bundle $f^*E^*$ is therefore convex Kobayashi negative, and $f^*E$ is convex Kobayashi positive.

Alternatively, one can use formula (\ref{local for koba}) to prove that the pull-back of a Kobayashi positive  Finsler metric is Kobayashi positive without passing to the dual.

\end{proof}

\bibliographystyle{amsalpha}
\bibliography{Dominion}

\providecommand{\bysame}{\leavevmode\hbox to3em{\hrulefill}\thinspace}
\providecommand{\MR}{\relax\ifhmode\unskip\space\fi MR }
\providecommand{\MRhref}[2]{%
  \href{http://www.ams.org/mathscinet-getitem?mr=#1}{#2}
}
\providecommand{\href}[2]{#2}
\begin{thebibliography}{Lem17b}

\bibitem[Aik98]{Aikoupartial}
Tadashi Aikou, \emph{A partial connection on complex {F}insler bundles and its
  applications}, Illinois J. Math. \textbf{42} (1998), no.~3, 481--492.
  \MR{1631260}

\bibitem[Aik04]{AikouMSRI}
\bysame, \emph{Finsler geometry on complex vector bundles}, A sampler of
  {R}iemann-{F}insler geometry, Math. Sci. Res. Inst. Publ., vol.~50, Cambridge
  Univ. Press, Cambridge, 2004, pp.~83--105. \MR{2132658}

\bibitem[Alb22]{albesiano2023deformation}
Roberto Albesiano, \emph{A deformation approach to {S}koda's division theorem},
  arXiv preprint arXiv:2212.07298 (2022).

\bibitem[BA02]{BenAbdesselem}
A.~Ben~Abdesselem, \emph{Complex {F}insler structures on tensor products}, J.
  Geom. Anal. \textbf{12} (2002), no.~4, 529--542. \MR{1916858}

\bibitem[BAA16]{BenAbdesselem2}
Adn\`ene Ben~Abdesselem and Ines Adouani, \emph{Complex {F}insler structures on
  tensor products}, Adv. Geom. \textbf{16} (2016), no.~1, 21--31. \MR{3451260}

\bibitem[Ber09]{Berndtsson09}
Bo~Berndtsson, \emph{Curvature of vector bundles associated to holomorphic
  fibrations}, Ann. of Math. (2) \textbf{169} (2009), no.~2, 531--560.
  \MR{2480611}

\bibitem[BL16]{BoLem}
Bo~Berndtsson and L\'{a}szl\'{o} Lempert, \emph{A proof of the
  {O}hsawa-{T}akegoshi theorem with sharp estimates}, J. Math. Soc. Japan
  \textbf{68} (2016), no.~4, 1461--1472. \MR{3564439}

\bibitem[CF90]{CampanaFlenner}
F.~Campana and H.~Flenner, \emph{A characterization of ample vector bundles on
  a curve}, Math. Ann. \textbf{287} (1990), no.~4, 571--575. \MR{1066815}

\bibitem[CS93]{CofimanSemmes}
R.~R. Coifman and S.~Semmes, \emph{Interpolation of {B}anach spaces, {P}erron
  processes, and {Y}ang-{M}ills}, Amer. J. Math. \textbf{115} (1993), no.~2,
  243--278. \MR{1216432}

\bibitem[CW03]{caowong}
J.-G. Cao and Pit-Mann Wong, \emph{Finsler geometry of projectivized vector
  bundles}, J. Math. Kyoto Univ. \textbf{43} (2003), no.~2, 369--410.
  \MR{2051030}

\bibitem[Dem99]{DemaillyMSRI}
Jean-Pierre Demailly, \emph{Pseudoconvex-concave duality and regularization of
  currents}, Several complex variables ({B}erkeley, {CA}, 1995--1996), Math.
  Sci. Res. Inst. Publ., vol.~37, Cambridge Univ. Press, Cambridge, 1999,
  pp.~233--271. \MR{1748605}

\bibitem[Dem12]{demailly1997complex}
\bysame, \emph{Complex analytic and differential geometry},
  https://www-fourier.ujf-grenoble.fr/~demailly/manuscripts/agbook.pdf, 2012.

\bibitem[Dem21]{demailly2020hermitianyangmills}
\bysame, \emph{Hermitian-{Y}ang-{M}ills approach to the conjecture of
  {G}riffiths on the positivity of ample vector bundles}, Mat. Sb. \textbf{212}
  (2021), no.~3, 39--53. \MR{4223969}

\bibitem[FLW16]{liuchern}
Huitao Feng, Kefeng Liu, and Xueyuan Wan, \emph{Chern forms of holomorphic
  {F}insler vector bundles and some applications}, Internat. J. Math.
  \textbf{27} (2016), no.~4, 1650030, 22. \MR{3491047}

\bibitem[FLW17]{liudonaldson}
\bysame, \emph{A {D}onaldson type functional on a holomorphic {F}insler vector
  bundle}, Math. Ann. \textbf{369} (2017), no.~3-4, 997--1019. \MR{3713533}

\bibitem[FLW20]{FengLiuWan}
\bysame, \emph{Complex {F}insler vector bundles with positive {K}obayashi
  curvature}, Math. Res. Lett. \textbf{27} (2020), no.~5, 1325--1339.
  \MR{4216589}

\bibitem[Gri69]{Griff69}
Phillip~A. Griffiths, \emph{Hermitian differential geometry, {C}hern classes,
  and positive vector bundles}, Global {A}nalysis ({P}apers in {H}onor of {K}.
  {K}odaira), Univ. Tokyo Press, Tokyo, 1969, pp.~185--251. \MR{0258070}

\bibitem[GW73]{GreeneWu2}
R.~E. Greene and H.~Wu, \emph{On the subharmonicity and plurisubharmonicity of
  geodesically convex functions}, Indiana Univ. Math. J. \textbf{22} (1973),
  641--653. \MR{422686}

\bibitem[GW79]{GreeneWu}
\bysame, \emph{{$C^{\infty }$} approximations of convex, subharmonic, and
  plurisubharmonic functions}, Ann. Sci. \'{E}cole Norm. Sup. (4) \textbf{12}
  (1979), no.~1, 47--84. \MR{532376}

\bibitem[Har66]{Hart66}
Robin Hartshorne, \emph{Ample vector bundles}, Inst. Hautes \'{E}tudes Sci.
  Publ. Math. (1966), no.~29, 63--94. \MR{193092}

\bibitem[HMP10]{toric}
Milena Hering, Mircea Mustata, and Sam Payne, \emph{Positivity properties of
  toric vector bundles}, Ann. Inst. Fourier (Grenoble) \textbf{60} (2010),
  no.~2, 607--640. \MR{2667788}

\bibitem[Kob75]{Negfinsler}
Shoshichi Kobayashi, \emph{Negative vector bundles and complex {F}insler
  structures}, Nagoya Math. J. \textbf{57} (1975), 153--166. \MR{377126}

\bibitem[Kob96]{ComplexFinsler}
\bysame, \emph{Complex {F}insler vector bundles}, Finsler geometry ({S}eattle,
  {WA}, 1995), Contemp. Math., vol. 196, Amer. Math. Soc., Providence, RI,
  1996, pp.~145--153. \MR{1403586}

\bibitem[Lem15]{LLmax}
L\'{a}szl\'{o} Lempert, \emph{A maximum principle for {H}ermitian (and other)
  metrics}, Proc. Amer. Math. Soc. \textbf{143} (2015), no.~5, 2193--2200.
  \MR{3314125}

\bibitem[Lem17a]{LLextrapolation}
\bysame, \emph{Extrapolation, a technique to estimate}, Functional analysis,
  harmonic analysis, and image processing: a collection of papers in honor of
  {B}j\"{o}rn {J}awerth, Contemp. Math., vol. 693, Amer. Math. Soc.,
  Providence, RI, 2017, pp.~271--281. \MR{3682614 (for a corrected version, see
  arxiv: 1507.06216v2)}

\bibitem[Lem17b]{LLnoncommutative}
\bysame, \emph{Noncommutative potential theory}, Anal. Math. \textbf{43}
  (2017), no.~4, 603--627. \MR{3738363}

\bibitem[LSY13]{positivityandvanishingthmliu}
Kefeng Liu, Xiaofeng Sun, and Xiaokui Yang, \emph{Positivity and vanishing
  theorems for ample vector bundles}, J. Algebraic Geom. \textbf{22} (2013),
  no.~2, 303--331. \MR{3019451}

\bibitem[MT07]{MourouganeTaka}
Christophe Mourougane and Shigeharu Takayama, \emph{Hodge metrics and
  positivity of direct images}, J. Reine Angew. Math. \textbf{606} (2007),
  167--178. \MR{2337646}

\bibitem[MZ23]{Mazhang}
Xiaonan Ma and Weiping Zhang, \emph{Superconnection and family {B}ergman
  kernels}, Math. Ann. \textbf{386} (2023), no.~3-4, 2207--2253. \MR{4612416}

\bibitem[Nau21]{naumann2017approach}
Philipp Naumann, \emph{An approach to the {G}riffiths conjecture}, Math. Res.
  Lett. \textbf{28} (2021), no.~5, 1505--1523. \MR{4471718}

\bibitem[Pin21]{pingali2021note}
Vamsi~Pritham Pingali, \emph{A note on {D}emailly's approach towards a
  conjecture of {G}riffiths}, C. R. Math. Acad. Sci. Paris \textbf{359} (2021),
  501--503. \MR{4278904}

\bibitem[Ric68]{Richberg}
Rolf Richberg, \emph{Stetige streng pseudokonvexe {F}unktionen}, Math. Ann.
  \textbf{175} (1968), 257--286. \MR{222334}

\bibitem[Roc84]{rochberg1984}
Richard Rochberg, \emph{Interpolation of {B}anach spaces and negatively curved
  vector bundles.}, Pacific J. Math. \textbf{110} (1984), no.~2, 355--376.

\bibitem[Sak32]{Saks}
S.~Saks, \emph{On subharmonic functions}, Acta Litt. Sci. Szeged 5 (1932),
  187--193.

\bibitem[Slo88]{SlodI}
Zbigniew Slodkowski, \emph{Complex interpolation of normed and quasinormed
  spaces in several dimensions. {I}}, Trans. Amer. Math. Soc. \textbf{308}
  (1988), no.~2, 685--711. \MR{951623}

\bibitem[Som78]{Sommese}
Andrew~John Sommese, \emph{Concavity theorems}, Math. Ann. \textbf{235} (1978),
  no.~1, 37--53. \MR{486637}

\bibitem[Ume73]{Umemura}
Hiroshi Umemura, \emph{Some results in the theory of vector bundles}, Nagoya
  Math. J. \textbf{52} (1973), 97--128. \MR{337968}

\bibitem[Wu20]{wuDirichlet}
Kuang-Ru Wu, \emph{A {D}irichlet problem in noncommutative potential theory},
  Michigan Math. J. \textbf{69} (2020), no.~2, 369--380. \MR{4104378}

\bibitem[Wu22a]{wu_2022}
\bysame, \emph{Positively curved {F}insler metrics on vector bundles}, Nagoya
  Math. J. \textbf{248} (2022), 766--778. \MR{4508264}

\bibitem[Wu22b]{wupositivelyII}
\bysame, \emph{Positively curved {F}insler metrics on vector bundles {II}},
  arXiv preprint arXiv:2210.12645, to appear in Pacific Journal of Mathematics
  (2022).

\bibitem[Wu23]{wuwess}
\bysame, \emph{A {W}ess-{Z}umino-{W}itten type equation in the space of
  {K}\"{a}hler potentials in terms of {H}ermitian-{Y}ang-{M}ills metrics},
  Anal. PDE \textbf{16} (2023), no.~2, 341--366. \MR{4593768}

\end{thebibliography}

\textsc{Department of Mathematics, National Tsing Hua University, Hsinchu, Taiwan}

\texttt{\textbf{krwu@math.nthu.edu.tw}}

\end{document}